\documentclass[11pt, a4paper]{article}
\usepackage{amsmath, amssymb, amsfonts, amsthm}
\usepackage[pdftex]{graphicx, color}

\newcommand{\E}{{\bf{E}}}
\newcommand{\PP}{{\bf{P}}}

\newtheorem{tm}{Theorem}
\newtheorem{lem}{Lemma}

\begin{document}

\parindent=0pt

\smallskip
\par\vskip 3.5em
\centerline{\Large \bf Connectivity threshold  for superpositions}
\centerline{\Large \bf   of Bernoulli random graphs. II}

\vglue2truecm

\vglue1truecm
\centerline {Mindaugas Bloznelis, Dominykas Marma, Rimantas Vaicekauskas}

\bigskip

\centerline{Institute of Computer Science, Vilnius University}
\centerline{
\, \ Didlaukio 47, LT-08303 Vilnius, Lithuania} 

\vglue2truecm

\abstract{Let $G_1,\dots, G_m$ be independent
Bernoulli random subgraphs of the complete graph ${\cal K}_n$ having
variable sizes $X_1,\dots, X_m\in \{0,1,2,\dots\}$  and densities $Q_1,\dots, Q_m\in [0,1]$. 
Letting $n,m\to+\infty$ we establish the connectivity threshold for the union 
$\cup_{i=1}^mG_i$ defined on the vertex set of ${\cal K}_n$. 
Assuming that $(X_1,Q_1), (X_2,Q_2),\dots, (X_m,Q_m)$
are 
 independent
 identically distributed bivariate random variables
 and $\ln n -\frac{m}{n}\E\bigl(X_1(1-(1-Q_1)^{|X_1-1|}\bigr)\to c$ 
 we show that
 $\PP\{\cup_{i=1}^mG_i$ is connected$\}\to e^{-e^c}$.
 The result extends to the case of non-identically  distributed random variables 
$(X_1,Q_1),\dots, (X_m,Q_m)$ as well.
}
\smallskip
\par\vskip 3.5em
\section{Introduction}

Let $G_1,\dots, G_m$ be independent Bernoulli random subgraphs of the complete graph ${\cal K}_n$ having (random) sizes $X_1,\dots, X_m$ and densities $Q_1,\dots, Q_m$. More precisely, given a sequence of independent bivariate random vectors $(X_1,Q_1),\dots, (X_m,Q_m)$,
each $G_k=({\cal V}_k,{\cal E}_k)$ is obtained by  selecting 
a 
subset of vertices ${\cal V}_k$ of ${\cal K}_n$ of size 
$|{\cal V}_k|=X_k$ 
and retaining edges between selected vertices independently at random with probability $Q_k$.  In particular, each $G_k$  is a random graph on $X_k$ vertices, where every pair of vertices is linked by an edge independently at random with probability $Q_k$.
The union (superposition)    
$G_1\cup\cdots\cup G_m$
represents a network of (possibly) overlapping communities  $G_1,\dots, G_m$. It was  called community 
affiliation graph by  Yang and Leskovec \cite{Yang_Leskovec2012, Yang_Leskovec2014} and  has received considerable attention in the literature. We note that  Yang and Leskovec \cite{Yang_Leskovec2012, Yang_Leskovec2014} defined community memberships (i.e. the sets ${\cal V}_1,\dots, {\cal V}_m$) by design.
  In this paper we 
 study the community affiliation graph 
 $G=\bigl({\cal V},\cup_{k=1}^m{\cal E}_k\bigr)$ 
  in the case where
 the vertex sets ${\cal V}_1$, $\dots$, ${\cal V}_m$  of the  communities $G_1,\dots, G_m$ are drawn independently at random from the vertex set
  ${\cal V}=[n]$ of ${\cal K}_n$. 
 We mention that in the range of parameters  $m=m(n)=\Theta(n)$
 the random graph $G$ 
 admits a power law  (asymptotic) degree distribution  and non-vanishing global clustering coefficient \cite{Lasse_Mindaugas2019}.
 The phase transition in the size of the largest component has been shown in  \cite{Lasse_Mindaugas2019}. Here we consider the parametric range $m=\Theta(n\ln n)$ and address the property of full connectivity.

The paper is organized as follows. In the remaining part of this section we  rigorously define the random graph model, present our results and give a brief overview of related previous work. Proofs are postponed to the next section.

\subsection{Random community affiliation graph}

Given $n$ and $m$, let $(X_{n,1},Q_{n,1}),\dots, (X_{n,m},Q_{n,m})$ be independent bivariate random variables taking values in 
$\{0,1,\dots, n\}\times [0,1]$.
For each $i\in[m]$, the $i$-th community 
$G_{n,i}=({\cal V}_{n,i},{\cal E}_{n,i})$ is defined as follows. 
We firstly sample $(X_{n,i},Q_{n,i})$. Then, given $(X_{n,i},Q_{n,i})$,
we select the vertex set ${\cal V}_{n,i}\subset [n]$ of size $|{\cal V}_{n,i}|=X_{n,i}$ uniformly at random and link every pair of elements of ${\cal V}_{n,i}$ independently at random with probability $Q_{n,i}$. We study   the union graph $G_{[n,m]}=({\cal V}, {\cal E})$ with  
the vertex set ${\cal V}=[n]$ and the edge set 
${\cal E}=\cup_{i\in[m]}{\cal E}_{n,i}$. We say that  $G_{[n,m]}$ is 
defined by the sequence 
$(X_{n,1},Q_{n,1}),\dots, (X_{n,m},Q_{n,m})$.  In Theorems \ref{Theorem_1} and \ref{Theorem_2} below we assume that $m=m(n)\to+\infty$ as $n\to+\infty$ and study  a sequence of random graphs $G_{[n,m(n)]}$, $n=1,2,\dots$.

Our first result, Theorem \ref{Theorem_1} establishes the connectivity threshold for $G_{[n,m]}$ in the case where the contributing random graphs 
$G_{n,1},\dots, G_{n,m}$ are identically distributed. Let  $(X,Q), (X_1,Q_1),(X_2,Q_2),\dots$ be a sequence of  independent and identically distributed bivariate random variables. For each $n=1,2,\dots$ and $i=1,\dots, m$ we put $X_{n,i}=\min\{X_i,n\}$.

Introduce the function 
\begin{displaymath}
h(x,q)=1-(1-q)^{(x-1)_{+}},
\end{displaymath}
defined for $(x,q)\in[0;+\infty)\times[0;1]$.
 Here we denote $(x)_+=\max\{x,0\}$ and assign value $1$ to the expression  $0^0$. By ${\mathbb I}_{\cal A}$ we denote the indicator function of the condition (event) ${\cal A}$. We denote
\begin{displaymath}
\alpha
=
\E(Q{\mathbb I}_{\{X\ge 2\}}),
\qquad
\varkappa
=
\E (Xh(X,Q)),
\qquad
\lambda_{m,n}
=
\ln n-\varkappa\frac{m}{n}.
\end{displaymath}

\begin{tm}\label{Theorem_1}
Let $c$ be a real number.
Let $n\to+\infty$. 
Assume that $m=m(n)\to +\infty$. 
Assume that 
$\E \left(Xh(X,Q)\ln (X+1)\right)<\infty$
and $\E(Q{\mathbb I}_{\{X\ge 2\}})>0$.
Then 
\begin{equation}
\label{Theorem_1+}
\lim_{n\to+\infty}
\PP\{G_{[n,m]}{\rm{\ is \ connected \ }}\}
=
\begin{cases}
1,
\qquad 
\
{\rm{for}}
\quad \lambda_{m,n}\to-\infty;
\\
e^{-e^c},
\quad 
{\rm{for}}
\quad \lambda_{m,n}\to c;
\\
0,
\qquad 
\
{\rm{for}}
\quad \lambda_{m,n}\to+\infty.
\end{cases}
\end{equation}
\end{tm}

We observe that inequality $h(x,q)\ge q$, which holds for $x\ge 2$, $q\in[0,1]$ implies
 \begin{displaymath}
 \varkappa
 \ge 
 \E (XQ{\mathbb I}_{\{X\ge 2\}})
 \ge 
  \E (Q{\mathbb I}_{\{X\ge 2\}}).
 \end{displaymath}
 Hence,  the condition $\E (Q{\mathbb I}_{\{X\ge 2\}})>0$ of Theorem \ref{Theorem_1} implies $\varkappa>0$.
 We mention that the integrability condition $\E \left(Xh(X,Q)\ln (X+1)\right)<\infty$ can not be relaxed, see  \cite{Daumilas_Mindaugas2023}.
\bigskip

Our next result, Theorem \ref{Theorem_2} addresses the case 
 where the community sizes and strengths 
$(X_{n,1},Q_{n,1}),\dots, (X_{n,m},Q_{n,m})$ are not necessarily identically distributed. In particular,  some or all of the members of the sequence can be non random  bivariate vectors.
Before formulating the result we introduce some notation. Let $i_*$ be a random number uniformly distributed in $[m]$. We assume that $i_*$ is independent of the sequence $\{(X_{n,i},Q_{n,i})$, $i\in [m]\}$. The bivariate random variable 
$(X_{n,i_*},Q_{n,i_*})$ has probability distribution,
\begin{displaymath}
\PP\bigr\{X_{n,i_*}=x, Q_{n,i_*}\in [0,t]\bigl\}
=\frac{1}{m}\sum_{i\in[m]}\PP\{X_{n,i}=x, Q_{n,i}\in[0,t]\}.
\end{displaymath}
We denote this probability distribution $P_{n,m}$. We will assume that $P_{n,m}$ converges weakly to a probability distribution, say, $P$ on $\{0,1,2,\dots\}\times [0,1]$. 
With a little abuse of notation we denote by $(X,Q)$ a bivariate random variable with the distribution $P$.
Denote 
\begin{displaymath}
\kappa_n=\E\bigl(X_{n,i_*}h(X_{n,i_*},Q_{n,i_*})\bigr)
=\frac{1}{m}\sum_{i\in[m]}\E\bigl(X_{n,i}h(X_{n,i},Q_{n,i})\bigr)
\end{displaymath}
and
\begin{displaymath}
\lambda^*_{n,m}=\ln n-\frac{m}{n}\kappa_n.
\end{displaymath}

\begin{tm}\label{Theorem_2} Let $m,n\to+\infty$.
 Assume that   $(X_{n,i_*},Q_{n,i_*})$ converges in distribution to some
  bivariate random random variable $(X,Q)$ such that 
 $\E\bigl(Q{\mathbb I}_{\{X\ge 2\}}\bigr)>0$  and 
\begin{equation}
\label{2023-04-29+12}
\E \bigl(Xh(X,Q)\ln (1+X)\bigr)<\infty.
\end{equation} 
Assume, in addition, that
\begin{equation}
\label{2023-06-03}
\E \bigl(X_{n,i_*}h(X_{n,i_*},Q_{n,i_*})\ln (1+X_{n,i_*})\bigr)
\to
\E \bigl(Xh(X,Q)\ln (1+X)\bigr).
\end{equation} 

Then 
\begin{equation}
\label{Theorem_2+}
\lim_{n\to+\infty}
\PP\{G_{[n,m]}{\rm{\ is \ connected \ }}\}
=
\begin{cases}
1,
\qquad 
\
{\rm{for}}
\quad \lambda^*_{m,n}\to-\infty;
\\
e^{-e^c},
\quad 
{\rm{for}}
\quad \lambda^*_{m,n}\to c;
\\
0,
\qquad 
\
{\rm{for}}
\quad \lambda^*_{m,n}\to+\infty.
\end{cases}
\end{equation}
\end{tm}

\subsection{Related previous work}
After the seminal work of Erd\H os and R\'enyi \cite{ErdosRenyi1959}  the connectivity thresholds for various  random graph models has attracted considerable attention in the literature and still remains an area of active research, see \cite{FriezeKaronski},
\cite{JansonLuczakRucinski2001}, \cite{Hofstad2017} and references therein.

An important property of the random graph model considered here is that in the (sparse) parametric regime $m=\Theta(n)$
it admits an asymptotic power-law degree distribution with a  tunable power-law exponent and non-vanishing global clustering coefficient,  \cite{Lasse_Mindaugas2019}.
It doesn't look surprising that the connectivity thershold (\ref{Theorem_1+}), (\ref{Theorem_2+}) occurs at the range $m=\Theta(n\ln n)$ (recall that the number of edges of the Erd\H{o}s-R\'enyi graph at the connectivity threshold  has the same order of magnitude). What is perhaps more interesting is that  
  the only characteristic that defines the  threshold is the mixed moment $\varkappa$ (respectively $\kappa_n$). 
For  a comparison,  the connectivity threshold of
  another class of inhomogeneous random graphs, where edges between vertices, say $u,v$, are inserted independently  with  probability, say $p_n(u,v)$, 
  is defined by 
  $\operatorname*{ess\,inf}$ of the properly scaled limit of $p_n$ as $n\to+\infty$, see \cite{DevroyeFraiman}. 

The present work complements the earlier study \cite{Daumilas_Mindaugas2023}, where the 
zero - one law for the connectivity property of 
the community affiliation graph has been 
established under the minimal moment conditions 
on the probability distribution of $(X,Q)$. 
Other previous  papers concerned with the connectivity threshold 
of the community affiliation graph have focused on the 
particular case where the contributing subgraphs 
(communities) $G_{n,i}$ were complete graphs 
(cliques). 
The connectivity threshold
(\ref{Theorem_1+}) has been reported in \cite{GodehardJaworskiRybarczyk2007}, where it was assumed 
 that $Q\equiv 1$ and 
$\exists M>0$ such that
$\PP\{X\le M\}=1$. We note that the approach adopted in  \cite{GodehardJaworskiRybarczyk2007} is different from ours. 
In the recent paper \cite{BergmanLeskela} the boundedness condition $\PP\{X\le M\}=1$ of \cite{GodehardJaworskiRybarczyk2007}  has been relaxed and 
the zero - one law for the connectivity property has been shown assuming the second moment condition $\E X^2<\infty$. 
The present paper gives a definite answer to the connectivity threshold problem for the community affiliation graph.

\section{Proofs}
In the proof we often use shorthand notation $G_i=G_{n,i}$ and ${\cal V}_i={\cal V}_{n,i}$, ${\cal E}_i={\cal E}_{n,i}$, $1\le i\le m$.

\subsection{Proof of Theorem \ref{Theorem_1}}

\begin{proof}[Proof of Theorem \ref{Theorem_1}]
In the case where $\lambda_{m,n}\to\pm\infty$ the zero - one law of (\ref{Theorem_1+})  is shown in  Proposition 1 of \cite{Daumilas_Mindaugas2023}.
Here we show (\ref{Theorem_1+}) in the case where $\lambda_{m,n}\to c$.

We note that $G$ is disconnected whenever it has a (connected)
component of size $k$ for some $1\le k\le n/2$.
Let $Y_k$ denote the number of connected components of size $k$ and 
${\cal A}_i=
\left\{\sum_{i\le k\le n/2}Y_k\ge 1\right\}$
denote the event that $G$ has a component of size $k$ for some $k\ge i$. The probability that $G$ is disconnected equals
\begin{displaymath}
\PP\{{\cal A}_1\}=\PP\{{\cal A}_2\cup\{Y_1\ge 1\}\}
=\PP\{Y_1\ge 1\}
+\PP\{{\cal A}_2\setminus \{Y_1=1\}\}.
\end{displaymath}
We show in Lemma 
\ref{Lemma_1} below that 
$\PP\{{\cal A}_2\}=o(1)$. In Lemma \ref{Lemma_2}
we show that the random variable $Y_1$ ($=$ the number of isolated vertices) converges in distribution to the Poisson distribution with mean value $e^c$ 
(distribution ${\cal P}(e^{c})$).
In this way we obtain the asymptotic relation
\begin{align*}
1- \PP\{G{\text{\ is connected}}\}
=
\PP\{{\cal A}_1\}  
&
=
\PP\{Y_1\ge 1\}+o(1)
\\
&
=
1-\PP\{Y_1=0\}+o(1)
\\
&
=
1-e^{-e^c}+o(1).
\end{align*}
Hence $\PP\{G{\text{\ is connected}}\}
=
e^{-e^c}+o(1)$.
\end{proof}

Before the formulation and proof of  Lemmas \ref{Lemma_1} and \ref{Lemma_2} we collect some notation and auxiliary results.
Given integer $x\ge 2$ and  real number $q\in[0,1]$, let $V_x\subset {\cal V}$ be a random subset of size $|V_x|=x$. Let $G_x=(V_x,{\cal E}_x)$ be Bernoulli random graph with the vertex set $V_x$ and with edge probability $q$. Let $q_k(x,q)$ denote the probability that none of the edges of $G_x$ connect subsets $[k]$ and $[n]\setminus[k]$,
\begin{displaymath}
q_k(x,q)=\PP\bigl\{\forall\{u,v\}\in{\cal E}_x
\quad {\text{either}}
\quad u,v\in [k] \quad {\text{or}}
\quad
u,v\in[n]\setminus[k]
\bigr\}.
\end{displaymath}
We put $q_k(x,q)=1$ for $x=0,1$. We will use the following result of \cite{Daumilas_Mindaugas2023}.

\begin{lem}\label{Daumilo-lema}
For $1\le k\le n/2$, $2\le x\le n$ and $0\le q\le 1$ we have
\begin{align}
\label{2023-05-03}
q_k(x,q)
&
\le 
1-2\frac{k(n-k)}{n(n-1)}q,
\\
\label{2023-05-01+1}
q_k(x,q)
&
\le 
1-\left(\frac{k}{n}x-R_1-R_2\right)h(x,q),
\end{align}
where $R_1=\frac{k^2}{(n-k)^2}$ 
and 
$R_2=R_2(k,x)=e^{-\frac{k}{n}x}-1+\frac{k}{n}x$.
\end{lem}

Let ${\hat q}_k$ denote the probability that
 $G_1$ has no edge with one endpoint in $[k]$ and another endpoint in 
$[n]\setminus[k]$. We denote for short 
${\tilde X}=\min\{n,X\}$ and ${\tilde X}_i=\min\{X_i,n\}$. We denote
\begin{displaymath}
{\tilde \varkappa}
=
\E \bigl({\tilde X}h({\tilde X},Q)\bigr),
\quad
{\tilde \lambda}_{m,n}
=
\ln n-{\tilde \varkappa}\frac{m}{n}.
\end{displaymath}
The following lemma is a simple consequence of Lemma \ref{Daumilo-lema}.

\begin{lem}\label{Lemma_3}
Let $c$ be a real number.
Assume that $\E(Xh(X,Q)\ln(X+1))<\infty$,
$\E(Q{\mathbb I}_{\{X\ge 2\}})>0$.
Assume that ${\tilde \lambda}_{m,n}\to c$ as $n\to+\infty$.
 Denote 
 $\delta_n
 =
 {\tilde \lambda}_{m,n}+\frac{2}{{\tilde \varkappa}}$ 
 and
 $\delta=2c+\frac{3}{\varkappa}$.
For any $0<\beta<1$ there exists $n_{\beta}>0$ depending on 
$\beta$ and on the probability distribution of 
$(X,Q)$ such that for 
$n>n_{\beta}$ we have 
\begin{equation}
\label{X16+1}
{\hat q}_k^{m}
\le 
e^{-k\ln n+k\delta_n},
\qquad
{\hat q}_k^{m-k}
\le 
e^{
-k\ln n
+
k\delta_n
+
\frac{k^2}{n}\ln n
},
\qquad
{\text{for}}
\qquad
1\le k\le n^{\beta}.
\end{equation}
We note that ${\tilde \lambda}_{m,n}\to c$ 
and 
${\tilde \varkappa}\to\varkappa$ 
implies $\delta_n\le \delta$ for sufficiently large $n$.\end{lem}
\begin{proof}[Proof of Lemma \ref{Lemma_3}]
To show (\ref{X16+1}) we combine inequality $1+x\le e^x$ with  the inequality
\begin{equation}
\label{X16}
{\hat q}_k
\le 
1-\frac{k}{n}{\tilde \varkappa}+\frac{k}{n}\frac{2}{\ln n},
\end{equation}
which holds for sufficiently large $n$. Inequality (\ref{X16}) is a simple consequence of Lemma \ref{Daumilo-lema}, see  also formula {\color{red}\bf (27)} of
\cite{Daumilas_Mindaugas2023}.
\end{proof}

\begin{lem}
\label{Lemma_1} 
Let $c$ be a real number.
Let $n\to+\infty$. 
Assume that 
$m=m(n)\to\infty$.
Assume that $\E(Xh(X,Q)\ln(X+1))<\infty$
and $\E(Q{\mathbb I}_{\{X\ge 2\}})>0$.
Assume that ${\tilde \lambda}_{m,n}\to c$.
Then
\begin{equation}
\label{component2}
\PP\{G_{[n,m]} \rm{\ has\  a \ component\  on} \  k\ \rm{vertices \ for \ some} \  2\le k\le n/2\}=o(1).
\end{equation}   
We remark that 
condition ${\tilde \lambda}_{m,n}\to c$ can be replaced by 
$\lambda_{m,n}\to c$.
\end{lem}

\begin{proof}[Proof of Lemma \ref{Lemma_1}]
In the proof we use the notation introduced in the proof of Theorem \ref{Theorem_1}. Before the proof we introduce some more notation and collect auxiliary results.
We denote 
 $\beta
 =
 \max\{1-\frac{\alpha}{2\varkappa}, \frac{1}{2}\}$. 
 For a subset $U\subset {\cal V}$ we denote by ${\cal B}_U$ the event that
$U$ induces a  connected component of 
$G_{[n,m]}$. 
Below we use the simple identity 
\begin{equation}
\label{Y}
Y_k=\sum_{U\subset V, |U|=k}{\mathbb I}_{{\cal B}_U}.
\end{equation}
We denote by ${\cal D}_U$  the event that
$G_{[n,m]}$ has no edges connecting vertex sets $U$ and 
${\cal V}\setminus U$. 
Note that 
$\PP\{{\cal B}_U\}\le \PP\{{\cal D}_U\}$.

We observe that the integrability condition 
$\E (Xh(X,Q)\ln(1+X))<\infty$ implies that there exists an increasing function 
$\phi(t)\uparrow+\infty$ as $t\to+\infty$
(depending on the distribution of $(X,Q)$) such that 
\begin{equation}
\E (X\phi(X)h(X,Q)\ln(1+X))<\infty.
\end{equation}
Moreover, one can choose $\phi(t)$ such that the function $t\to t/(\phi(t)\ln(1+t))$ was increasing for $t=1,2,\dots$ and $\phi(1)=1$.

Furthermore we  observe that 
$\zeta:=\E(X\phi(X)Q\ln(1+X))<\infty$. Indeed, the inequality $q\le h(x,q)$, which holds for $x\ge 2$, implies 
\begin{equation}
\label{U_1}
\E(X\phi(X)Q\ln(1+X)) 
\le
\PP\{X=1\}\ln 2
+
\E(X\phi(X){\mathbb I}_{\{X\ge 2\}}h(X,Q)\ln(1+X)) 
<\infty
\end{equation}
 (here we use the fact that  $X$ is an integer valued random variable). The fact that 
 $t\to t/(\phi(t)\ln (1+t))$ is increasing for $t\ge 1$ implies 
$t/(\phi(t)\ln(1+t))\le n/(\phi(n)\ln(1+n))$, for $t=1,2,\dots, n$. Using this inequality we estimate for $a=1,2,\dots$,
\begin{align*}
{\tilde X}^{a+1}
&
=
\frac{{\tilde X}}
{\phi({\tilde X})\ln (1+\tilde X)}
{\tilde X}^{a}\phi({\tilde X})\ln(1+{\tilde X})\\
&
\le 
\frac{n}{\phi(n)\ln(1+n)} {\tilde X}^a
\phi({\tilde X})\ln(1+\tilde X)
\\
&
\le 
\frac{n^a}{\phi(n)\ln(1+n)} {\tilde X}
\phi({\tilde X})\ln(1+\tilde X).
\end{align*}
The latter inequality implies the bound   
 \begin{equation}
 \label{XQ}
\E({\tilde X}^{a+1}Q^a)
\le 
\frac{n^a}{\phi(n)\ln (1+n)}
\zeta,
\qquad a=1,2,\dots, 
\end{equation}
 which we are going to use in the proof of (\ref{component2}).
 
 \bigskip
 Now we are ready to prove Lemma \ref{Lemma_1}. Let us show that 
 ${\tilde \lambda}_{n,m}\to c \Leftrightarrow  \lambda_{n,m}\to c$. Note that each of these relations yields $m/n=O(\ln n)$. 
 Furthermore, we have  $\lambda_{m,n}-{\tilde \lambda}_{m,n}= (\varkappa-{\tilde \varkappa})m/n$ and 
 \begin{eqnarray}
 \nonumber
 0
 \le
  \varkappa-{\tilde \varkappa}
&&
 \le
 \E\left(Xh(X,Q){\mathbb I}_{\{X>n\}}\right)
\\
\nonumber
&&
 \le
 \frac{1}{\ln(1+n)}
  \E\left(Xh(X,Q){\mathbb I}_{\{X>n\}}\ln(1+X)\right)
 \\
 \label{varkappa}
 &&
 =o\left(\frac{1}{\ln n}\right).
 \end{eqnarray}

 We conclude that each of the relations ${\tilde \lambda}_{n,m}\to c$ and  $\lambda_{n,m}\to c$ 
 implies
 $\lambda_{m,n}-{\tilde \lambda}_{m,n}=o(1)$. Hence 
  ${\tilde \lambda}_{n,m}\to c \Leftrightarrow  \lambda_{n,m}\to c$.
 
 \bigskip
 
 Let us show (\ref{component2}). 
By the union bound and (\ref{Y}), we have
 \begin{displaymath}
 \PP\{{\cal A}_2\}
 \le \sum_{2\le k\le n/2}\PP\{Y_k\ge 1\}
\le
 \sum_{2\le k\le n/2}
 \
 \sum_{U\subset {\cal V}, |U|=k}\PP\{{\cal B}_U\}
 =
 \sum_{2\le k\le n/2}\binom{n}{k}\PP\{{\cal B}_{[k]}\}.
 \end{displaymath}
 In the last step we used the fact that 
 $\PP\{{\cal B}_U\}=\PP\{{\cal B}_{[k]}\}$ for 
 $|U|=k$.
 Next, using the inequality 
 $\PP\{{\cal B}_{[k]}\}
 \le
  \PP\{{\cal D}_{[k]}\}$ 
  we upper bound
 $\PP\{{\cal A}_2\}\le S_0+S_1+S_2$, 
 where 
 \begin{displaymath}
 S_0=\sum_{2\le k\le \varphi_n}
 \binom{n}{k}
 \PP\{{\cal B}_{[k]}\},
\qquad
 S_1=\sum_{\varphi_n< k\le n^{\beta}}
 \binom{n}{k}\PP\{{\cal D}_{[k]}\},
 \qquad
 S_2=\sum_{n^{\beta}\le k\le n/2}
 \binom{n}{k}\PP\{{\cal D}_{[k]}\}.
  \end{displaymath}
Here $\beta=\max\left\{1-\frac{\alpha}{2{\varkappa}},\frac{1}{2}\right\}$.
The sequence $\varphi_n\to+\infty$  as 
$n\to+\infty$ 
will be specified latter. To prove  (\ref{component2}) we
 show that
 $S_i=o(1)$ for $i=0,1,2$.

\bigskip

{\it Proof of} $S_0=o(1)$.
Let us evaluate the probabilities 
$\PP\{{\cal B}_{[k]}\}$, $k\ge 2$.
Let $T=(V_T,{\cal E}_T)$ be a tree with the vertex set 
$V_T=[k]\subset {\cal V}$. 
 Fix an integer $1\le r\le k-1$.
Let 
${\tilde {\cal E}}_T=({\cal E}_T^{(1)},\dots, {\cal E}_T^{(r)})$ be an ordered partition of the set ${\cal E}_T$ 
(every set ${\cal E}_T^{(i)}$ is nonempty, 
${\cal E}_T^{(i)}\cap{\cal E}_T^{(j)}=\emptyset$ 
for $i\not=j$, 
and $\cup_{i=1}^r{\cal E}_T^{(i)}={\cal E}_T$).
Let ${\tilde t}=(t_1,\dots, t_r)\in [m]^r$ be a 
vector with integer valued coordinates satisfying $t_1<t_2<\cdots<t_r$. Given a pair
$({\tilde {\cal E}}_T, {\tilde t})$,
 let ${\cal T}({\tilde {\cal E}}_T, {\tilde t})$ denote the event 
 that ${\cal E}_T^{(i)}\subset {\cal E}_{t_i}$ 
 for each $1\le i\le r$. The event means that the  edges of
$T$ are covered by  the edges of $G_{t_1},\dots G_{t_r}$ so 
that
for every $i$ the edge set ${\cal E}_T^{(i)}$ belongs to the edge set ${\cal E}_{t_i}$ of $G_{t_i}$.
Let $H_{\tilde t}=[m]\setminus\{t_1,\dots, t_r\}$  and let 
${\cal I}(V_T, H_{\tilde t})$ be the event that 
none of the graphs $G_i$, $i\in H_{\tilde t}$ has an edge connecting some 
$v\in V_T$ and $w\in {\cal V}\setminus V_T$.

\medskip

Let ${\mathbb T}_k$ denote the set of trees on the vertex set $[k]$.
We have, by the union bound and independence 
of $G_1,\dots, G_m$, that
\begin{equation}
\label{suma}
\PP\{{\cal B}_{[k]}\}
\le
\sum_{T\in {\mathbb T}_k}
\sum_{({\tilde {\cal E}}_T,{\tilde t})}
\PP\{{\cal T}({\tilde {\cal E}}_T, {\tilde t})\}
\PP\{{\cal I}(V_T, H_{\tilde t})\}.
\end{equation}
Here the second sum runs over all possible pairs 
$({\tilde {\cal E}}_T, {\tilde t})$. Invoking the bound 
$\PP\{
{\cal I}(V_T, H_{{\tilde t}})\}\le e^{-k\ln n+k\delta+1}$,
see the second inequality of Lemma 
\ref{Lemma_3}, which holds for sufficiently 
large $n$ uniformly over 
${\tilde t}=(t_1,\dots, t_r)$ with $r\in [1,k]$ 
satisfying $k\le \ln n$, we obtain
\begin{equation}
\label{suma++++}
\PP\{{\cal B}_{[k]}\}
\le
e^{-k\ln n+k\delta+1}
\sum_{T\in {\mathbb T}_k}
S_T,
\qquad 
S_T
:=
\sum_{({\tilde {\cal E}}_T,{\tilde t})}
\PP\{{\cal T}({\tilde {\cal E}}_T, {\tilde t})\}
.
\end{equation}
We show below that
\begin{equation}
\label{S_T}
S_T
\le
\frac{1}{\phi(n)}
\sum_{r=1}^{k-1}
\left(\frac{2}{\varkappa}\right)^r
\zeta^re^rr^{k-1-r}.
\end{equation}
Invoking this inequality in (\ref{suma++++}) and using Cayley's formula $|{\mathbb T}_k|=k^{k-2}$ we obtain
\begin{displaymath}
\PP\{{\cal B}_{[k]}\}
\le
e^{-k\ln n+k\delta+1}
\frac{R_k}{\phi(n)},
\qquad
R_k:=
k^{k-2}\sum_{r=1}^{k-1}
\left(\frac{2}{\varkappa}\right)^r
\zeta^re^rr^{k-1-r}.
\end{displaymath}
Furthermore, using the simple bound $\binom{n}{k}\le \frac{n^k}{k!}=\frac{e^{k\ln n}}{k!}$ we estimate
\begin{displaymath}
S_0
\le 
\frac{1}{\phi(n)}
\sum_{k=2}^{\varphi_n}
e^{k\delta+1}\frac{R_k}{k!}.
\end{displaymath}
Finally,  choosing a nondecreasing (integer valued) sequence $\varphi_n\to+\infty$ as $n\to+\infty$ such that $\varphi_n\le \ln n$ and 
$\sum_{k=2}^{\varphi_n}
e^{k\delta+1}\frac{R_k}{k!}
=o(\phi(n))$ we obtain $S_0=o(1)$.

\bigskip

{\it Proof of (\ref{S_T}).}
Given a tree $T=([k],{\cal E}_T)$ and  partition
${\tilde {\cal E}}_T=({\cal E}_T^{(1)},\dots,{\cal E}_T^{(r)})$, let
$V_T^{(i)}$ be the set of vertices incident to the edges from 
${\cal E}_T^{(i)}$. We denote $e_i=|{\cal E}_T^{(i)}|$ and
$v_i=|V_T^{(i)}|$.
For any labeling 
${\tilde t}=(t_{1},\dots, t_{r})$ that assigns labels 
$t_1,\dots, t_r$ to the sets
${\cal E}_T^{(1)},\dots,{\cal E}_T^{(r)}$ we have, by the independence of $G_1,\dots, G_m$,
\begin{displaymath}
\PP\{{\cal T}({\tilde {\cal E}}_T, {\tilde t})\}
=
\prod_{i=1}^r
\E
\left(
\frac{({\tilde X}_{t_i})_{v_i}}{(n)_{v_i}}Q_{t_i}^{e_i}
\right).
\end{displaymath}
We note that the fraction $\frac{({\tilde X}_{t_i})_{v_i}}{(n)_{v_i}}$ is 
a decreasing function of $v_i$ and it is maximized by 
$\frac{({\tilde X}_{t_i})_{e_i+1}}{(n)_{e_i+1}}$ since
 we always have $v_i\ge e_i+1$. Indeed,  
 given $|{\cal E}_T^{(i)}|=e_i$ the smallest possible set of vertices
$V_T^{(i)}$ 
corresponds to the configuration of edges of 
${\cal E}_T^{(i)}$ that creates a subtree $(V_T^{(i)},{\cal E}_T^{(i)})\subset T$. Hence $v_i\ge e_i+1$. It follows that 
\begin{equation}
\label{Stirling}
\PP\{{\cal T}({\tilde {\cal E}}_T, {\tilde t})\}
\le
\prod_{i=1}^r
\E
\left(
\frac{({\tilde X}_{t_i})_{e_i+1}}{(n)_{e_i+1}}Q_{t_i}^{e_i}
\right)
\le 
\left(\frac{\zeta}{n\phi(n)\ln(1+n)}\right)^r.
\end{equation}
In the last step we used (\ref{XQ}).

Let us evaluate the sum
$S_T$.
For $1\le r\le  k-1$ and  
a vector $(e_1,\dots, e_r)$ with integer valued coordinates  satisfying $e_1+\cdots+e_r=k-1$ and $e_i\ge 1$ $\forall i$,
there are 
$\frac{(k-1)!}{e_1!\cdots e_r!}$
ordered
partitions 
${\tilde {\cal E}}_T=({\cal E}_T^{(1)},\dots,{\cal E}_T^{(r)})$ 
of ${\cal E}_T$
in $r$ non empty parts of sizes 
$|{\cal E}_T^{(1)}|=e_1,\dots, |{\cal E}_T^{(r)}|=e_r$. 
Therefore
we have
 \begin{align*}
 S_T
&
\le
\sum_{r=1}^{k-1}
\binom{m}{r}
\sum'_{e_1+\dots+e_r=k-1}
\frac{(k-1)!}{e_1!\cdots e_r!}
\
\left(\frac{\zeta}
{n\phi(n)\ln(1+n)}\right)^r
\\
&
\le
\sum_{r=1}^{k-1}
\binom{m}{r}
\left(\frac{\zeta}
{n\phi(n)\ln(1+n)}\right)^rr^{k-1}.
\end{align*}
Here the sum $\sum'_{e_1+\dots+e_r=k-1}$ runs over the set of vectors
$(e_1,\dots, e_r)$ having integer valued coordinates $e_i\ge 1$ 
satisfying $e_ 1+\cdots+e_r=k-1$.

Invoking inequalities $\binom{m}{r}\le \left(\frac{me}{r}\right)^r$ and 
$\frac{m}{n\ln(1+n)}\le \frac{2}{\varkappa}$ (the last inequality holds for sufficiently large $n$ and $m$) we arrive to (\ref{S_T})
\begin{displaymath}
S_T
\le
\frac{1}{\phi(n)}
\sum_{r=1}^{k-1}
\left(\frac{m}{n\ln(1+n)}\right)^r
\zeta^re^rr^{k-1-r}
\le
\frac{1}{\phi(n)}
\sum_{r=1}^{k-1}
\left(\frac{2}{\varkappa}\right)^r
\zeta^re^rr^{k-1-r}.
\end{displaymath}

\bigskip

{\it Proof of} $S_1=o(1)$. The first inequality of Lemma \ref{Lemma_2} implies
$\PP\{{\cal D}_{[k]}\}\le e^{-k\ln n+k\delta}$.
Invoking the inequality 
$\binom{n}{k}\le \frac{n^k}{k!}$ we obtain
\begin{displaymath}
S_1
\le 
\sum_{\varphi_n<k\le n^{\beta}}
\frac{e^{k\delta}}{k!}
\le 
\sum_{k>\varphi_n}
\frac{e^{k\delta}}{k!}.
\end{displaymath}
For $n\to+\infty$ we have $S_1=o(1)$ because
$\varphi_n\to+\infty$ and the series 
$\sum_{k\ge 1}
\frac{e^{k\delta}}{k!}$ converges.

\bigskip
{\it Proof of} $S_2=o(1)$.
The proof is much the same as that of the corresponding bound in the proof of Proposition 1 of \cite{Daumilas_Mindaugas2023}. In particular, we use the inequality
$\binom{n}{k}\le e^{2k+(1-\beta)k\ln n}$, which holds for $n^{\beta}\le k\le n/2$,  and the inequality 
\begin{displaymath}
{\hat q}_k\le 1-\alpha \frac{k}{n},
\qquad
1\le k\le n/2.
\end{displaymath}
Both inequalities are shown in the proof of  
Proposition 1 of \cite{Daumilas_Mindaugas2023}.
Combining these inequalities and $1+x\le e^x$ we estimate the product
\begin{equation}
\label{2023-10-16}
\binom{n}{k}\PP\{{\cal D}_k\}
=
\binom{n}{k}{\hat q}_k^m
\le e^{2k+(1-\beta)k\ln n-\alpha k\frac{m}{n}}
=
e^{k\left(2+\alpha\frac{{\tilde \lambda}_{m,n}}{{\tilde \varkappa}}\right)
-k\left(\frac{\alpha}{\tilde\varkappa}+\beta-1\right)\ln n}.
\end{equation}
In the last step we used 
$\frac{m}{n}=\frac{1}{\tilde\varkappa}(\ln n-{\tilde \lambda}_{m,n})$.
For $\lambda_{m,n}\to c$ and $\varkappa-{\tilde \varkappa}=o\left(\frac{1}{\ln n}\right)$, see  (\ref{varkappa}), 
the quantity on the right
of (\ref{2023-10-16}) becomes
$e^{-k\left(\frac{\alpha}{\varkappa}+\beta-1+o(1)\right)\ln n}$.
Since our choice of $\beta$ yields
$\frac{\alpha}{\varkappa}+\beta-1\ge \frac{\alpha}{2\varkappa}$, we have
$S_2
\le 
\sum_{k\ge n^{\beta}}
e^{-k\left(\frac{\alpha}{2\varkappa}+o(1)\right)\ln n}=o(1)$.

\end{proof}

\bigskip
\begin{lem}\label{Lemma_2}
Let $c$ be a real number. Let $\Lambda$ be a random variable having  the Poisson distribution with mean value $e^{c}$.
Assume that conditions of Lemma \ref{Lemma_1} hold. Then the number of vertices of degree $0$ in $G_{[n,m]}$ converges in distribution to
$\Lambda$.
\end{lem} 

\begin{proof}[Proof of Lemma \ref{Lemma_2}]
Let $Y_0$ denote the number of vertices $v_i$ of $G_{[n,m]}$ 
having degree $d(v_i)=0$. We apply the method of moments to show that $Y_0$ converges in distribution to $\Lambda$. 
To this aim we establish the convergence of the factorial moments
as $n\to+\infty$
\begin{equation}
\label{2023-10-21}
\E(Y_0)_r\to \E (\Lambda)_r,
\qquad
{\text{for}}
\qquad r=1,2,\dots.
\end{equation}

 Note that $\E(\Lambda)_r=e^{rc}$.
Let us evaluate $\E(Y_0)_r$. In the proof we use  the identity
\begin{displaymath}
\binom{Y_0}{r}
=
\sum_{\{v_{i_1},\dots, v_{i_r}\}\subset {\cal V}}
{\mathbb I}_{\{d(v_{i_1})=0\}}\cdots{\mathbb I}_{\{d(v_{i_r})=0\}},
\end{displaymath}
where  both sides count the (same) number of subsets of size $r$ of vertices having degree $0$. Taking the expected values we obtain, by symmetry,
\begin{align}
\nonumber
\E (Y_0)_r
&
=r!\E\left( \sum_{\{v_{i_1},\dots, v_{i_r}\}\subset {\cal V}}{\mathbb I}_{\{d(v_{i_1})=0\}}\cdots{\mathbb I}_{\{d(v_{i_r})=0\}}\right)
\\
\label{2023-10-24+2}
&
=
r!\binom{n}{r}\PP\{d(v_1)=0,\dots, d(v_r)=0\}.
\end{align}
Furthermore, by the independence of $G_1,\dots, G_m$, we have 
\begin{displaymath}
\PP\{d(v_1)=0,\dots, d(v_r)=0\}
=P_r^m,
\qquad 
{\text{where}}
\qquad
P_r:=\PP\{d_1(v_1)=0,\dots, d_1(v_r)=0\}.
\end{displaymath}
Here $d_1(v)$ denotes the  degree of $v$ in  
$G_1=({\cal V}_1,{\cal E}_1)$  and we put $d_1(v)=0$ for $v\notin {\cal V}_1$.  

Next we approximate 
$P_r=1-\PP\{\cup_{i=1}^r\{d_1(v_i)>0\}\}$
using inclusion-exclusion inequalities
\begin{equation}
\label{X19}   
S_1-S_2 \le 1-P_r\le  S_1,
\end{equation}
where
\begin{align*}
 S_1
 &
 =\sum_{i=1}^r\PP\{d_1(v_i)>0\}=r\PP\{d_1(v_1)>0\},
\\
S_2
&
=
\sum_{1\le i<j\le r}
\PP\{d_1(v_i)>0,\, d_1(v_j)>0\}
=
\binom{r}{2}
\PP\{d_1(v_1)>0,\, d_1(v_2)>0\}.
\end{align*}
A straightforward calculation shows that
\begin{equation}\label{X19+}
 \PP \{ d_1(v_1) > 0 \}
=
\E
\left(
\frac{{\tilde X}_1}{n}
\bigl(1 - (1 - Q_i)^{({\tilde X}_1 - 1)_+}
\bigr)
\right)
=
\frac{{\tilde \varkappa}}{n}.
\end{equation}
Here $\frac{{\tilde X}_1}{n}$ accounts for  the probability of the event $v_1\in {\cal V}_1$ and $(1 - (1 - Q_i)^{({\tilde X}_1 - 1)_+}$ accounts for the probability that  given vertex of $G_1$  is not isolated in $G_1$.
Similarly, we have
\begin{displaymath}
\PP \{d_1(v_1) > 0 \cap d_1(v_2) > 0 \}
= 
\E 
\left(
\frac{({\tilde X}_1)_2}{(n)_2}
h_1({\tilde X}_1,Q_1)
\right),
\end{displaymath}
where
$h_1({\tilde X}_1,Q_1)
=
Q_1 + (1 - Q_1)\bigl(1 - (1 - Q_1)^{({\tilde X}_1 - 2)_+}\bigr)^2$
accounts for the pobability that  given a pair of vertices  of $G_1$  none of them is isolated in $G_1$. $\frac{({\tilde X}_1)_2}{(n)_2}$ accounts for  the probability that $\{v_1,v_2\}\subset {\cal V}_1$.
Invoking the inequality $h_1(x,q)\le h(x,q)$, which holds for $x\ge 2$, 
\begin{align*}
h_1(x,q)
&
=
q+(1-q)\Bigl(1-2(1-q)^{(x-2)_+}+(1-q)^{2(x-2)_+}\Bigr)
\\
&
=
1-2(1-q)^{1+(x-2)_+} +(1-q)^{1+2(x-2)_+}
\\
&
=
1-(1-q)^{1+(x-2)_+} -\Bigl((1-q)^{1+(x-2)_+} -(1-q)^{1+2(x-2)_+}
\Bigr)
\\
&
\le 
1-(1-q)^{1+(x-2)_+}
=
1-(1-q)^{(x-1)_+}
=
h(x,q),
\end{align*}
(we only need $x\ge 2$ in the second last identity) we obtain
\begin{equation}
\label{2023-10-24+21}
\PP \{d_1(v_1) > 0 \cap d_1(v_2) > 0 \}
\le 
\E 
\left(
\frac{({\tilde X}_1)_2}{(n)_2}
h({\tilde X}_1,Q_1)
\right)
=:\Delta.
\end{equation}

We show below that $
\Delta=o\left(\frac{1}{n\ln n}\right)$. This bound implies $S_2=o\left(\frac{r^2}{n\ln n}\right)$ and, together with
 (\ref{X19}) and (\ref{X19+}), it yields the approximation
\begin{displaymath}
P_r=1-{\tilde{\varkappa}}\frac{r}{n}+o\left(\frac{r^2}{n\ln n}\right).
\end{displaymath}
 
 Finally, using $(n)_r=(1+o(1))e^{r\ln n}$ 
 and 
 $1-x=e^{\ln(1-x)}=e^{-x+O(x^2)}$ 
 as 
 $x\to 0$ 
 we obtain
\begin{displaymath}
\E(Y_0)_r
=
(n)_rP_r^m
=
(1+o(1))e^{r\ln n}e^{-r{\tilde \varkappa}\frac{m}{n}+o(1)}
=
(1+o(1))e^{r{\tilde \lambda}_{m,n}+o(1)}.
\end{displaymath}
Now (\ref{2023-10-21}) follows by our assumption 
${\tilde \lambda}_{m,n}\to c$.

\bigskip

It remains to show that
$\Delta=o\left(\frac{1}{n\ln n}\right)$.
To this aim we show that 
$\E ({\tilde X}_1^2h({\tilde X}_1,Q_1))=o\left(\frac{n}{\ln n}\right)$.
We split
\begin{displaymath}
\E ({\tilde X}_1^2h({\tilde X}_1,Q_1))
=
\E (X_1^2h(X_1,Q_1){\mathbb I}_{\{X_1\le n\}})
+
n^2 \E (h(n,Q_1){\mathbb I}_{\{X_1> n\}})
=:
\Delta_1
+
n^2\Delta_2
\end{displaymath}
and treat $\Delta_1$ and $\Delta_2$ separately. By the fact that $t\to t/(\phi(t)\ln (1+t))$ is increasing,
\begin{align*}
\Delta_1
&
=
\E
\left(
\frac{X_1}{\phi(X_1)\ln(1+X_1)}
X_1\phi(X_1)
\ln(1+X_1)
{\mathbb I}_{\{X_1\le n\}}\right)
\\
&
\le
\frac{n}{\phi(n)\ln(1+n)}
\E (X_1\phi(X_1)\ln(1+X_1){\mathbb I}_{\{X_1\le n\}})
\\
&
\le \frac{\gamma}{\phi(n)}\frac{n}{\ln(1+n)}.
\end{align*}
Here $\gamma:=\E (X_1\phi(X_1)\ln(1+X_1))<\infty$. Now $\phi(n)\to+\infty$ implies $\Delta_1=o\left(\frac{n}{\ln n}\right)$.
Furthermore, by the fact that $t\to t\ln(1+t)$ is increasing and 
$t\to h(t,q)$ is nondecreasing, we have
\begin{align*}
\Delta_2
\le
\E\left(\frac{X_1\ln(1+X_1)}{n\ln(1+n)}h(X_1,Q_1){\mathbb I}_{\{X_1>n\}}\right)
=o\left(\frac{1}{n\ln n}\right).
\end{align*}
In the  last step we used the fact that $\gamma<\infty$ implies 
$\E\left(X_1\ln(1+X_1)h(X_1,Q_1){\mathbb I}_{\{X_1>n\}}\right)=o(1)$
as $n\to+\infty$. We conclude that 
$n^2\Delta_2=o\left(\frac{n}{\ln n }\right)$.
\end{proof}

\subsection{Proof of Theorem \ref{Theorem_2}}
Before the proof we introduce some notation. 
Let ${\hat q}_k^{[n,i]}$ denote the probability that
 $G_{n,i}$ has no edge with one endpoint in 
 $[k]$ and another endpoint in 
$[n]\setminus[k]$.
We denote
\begin{align*}
&
\kappa_{n,i}=\E\bigl(X_{n,i}h(X_{n,i},Q_{n,i})\bigr),
\qquad
\alpha_{n,i}
=
\E\bigl(Q_{n,i}{\mathbb I}_{\{X_{n,i}\ge 2\}}\bigr),
\\
&
\alpha_{n}=\E\bigl(Q_{n,i_*}{\mathbb I}_{\{X_{n,i_*}\ge 2\}}\bigr),
\qquad
\qquad
\
\,
\alpha
=
\E\bigl(Q{\mathbb I}_{\{X\ge 2\}}\bigr)
\end{align*}
and  observe that
\begin{displaymath}
\sum_{i=1}^m\alpha_{n,i}=m\alpha_n,
\qquad
\sum_{i=1}^m\kappa_{n,i}=m\kappa_n.
\end{displaymath}
Furthermore, we observe that the weak convergence  
$P_{n,m}\to P$ together with the convergence of moments (\ref{2023-06-03}) implies  
\begin{equation}\label{2023-10-28}
\alpha_n\to \alpha
\ \bigl(=
\E(Q{\mathbb I}_{\{X\ge 2\}})
\bigr),
\qquad
\kappa_n\to\varkappa
\ \bigl(=
\E (Xh(X,Q))
\bigr)
\qquad
{\text{as}}
\qquad
n\to+\infty.
\end{equation} 
Moreover,
the sequence 
of random variables 
$\{X_{n,i_*}h(X_{n,i_*},Q_{n,i_*})\ln(1+X_{n,i_*}), \, n\ge 1\}$ 
is uniformly integrable. That is, we have
\begin{displaymath}
\lim_{t\to+\infty}
\sup_{n}
\E
\left( 
X_{n,i_*}h(X_{n,i_*},Q_{n,i_*})\ln(1+X_{n,i_*})
{\mathbb I}_{\{X_{n,i_*}\ge t\}}
\right)
=
0.
\end{displaymath}
 Using the uniform integrability property above we can construct 
 a positive increasing function $\phi(x)$ such that 
$\lim_{x\to+\infty}\phi(x)=+\infty$ and
the sequence 
of random variables 
\linebreak
$\{X_{n,i_*}h(X_{n,i_*},Q_{n,i_*})\phi(X_{n,i_*})\ln(1+X_{n,i_*}), \, n\ge 1\}$ 
is uniformly integrable, see  \cite{Daumilas_Mindaugas2023}. We denote 
\begin{align*}
\varphi(t)
=
\sup_{n}
\E
\left( 
X_{n,i_*}h(X_{n,i*},Q_{n,i_*})\phi(X_{n,i_*})\ln(1+X_{n,i_*})
{\mathbb I}_{\{X_{n,i_*}\ge t\}}
\right).
\end{align*}
The  uniform integrability implies  $\varphi(t)\to 0$ as $t\to+\infty$.
Moreover,  one can choose $\phi(x)$ such that
 the function 
$x\to x/(\phi(x)\ln (1+x))$ 
was increasing for $x=1,2,\dots$, 
see \cite{Daumilas_Mindaugas2023}.  
Furthermore, the same argument as in (\ref{U_1}) shows that
\begin{displaymath}
\zeta_*
:
=
\sup_{n}
\E
\left( 
X_{n,i_*}Q_{n,i_*}\phi(X_{n,i_*})\ln(1+X_{n,i_*})
\right)
<\infty.
\end{displaymath}

In the proof below we use the bounds
 shown in \cite{Daumilas_Mindaugas2023} 
\begin{align}
\label{2023-10-26+1}
&
\E\bigl(X_{n,i_*}^2h(X_{n,i_*},Q_{n,i_*})\bigr)
=
o\left(\frac{n}{\ln n}\right),
\\
\label{2023-10-25}
&
\max_{1\le i\le n}\kappa^2_{n,i}
\le
\sum_{i\in[m]}\kappa^2_{n,i}
\le 
m \E (X^2_{n,i_*}h(X_{n,i_*},Q_{n,i_*}))
=o\left(\frac{mn}{\ln n}\right).
\end{align}
We note that 
for $m=O(n\ln n)$ the quantity on the right of (\ref{2023-10-25})
is $o(n^2)$.

\begin{proof}[Proof of Theorem \ref{Theorem_2}]
In the case where $\lambda^*_{m,n}\to\pm\infty$ the zero - one law of (\ref{Theorem_2+})  is shown in  Theorem 1 of \cite{Daumilas_Mindaugas2023}.
It remains to show (\ref{Theorem_2+}) in the case where $\lambda^*_{m,n}\to c$.
In this case we proceed similarly as in the proof of Theorem \ref{Theorem_1}. The only difference is that now we use Lemmas \ref{Lemma_1*} and \ref{Lemma_2*} (shown below) instead of Lemmas \ref{Lemma_1} and \ref{Lemma_2}.
\end{proof}

\begin{lem}
\label{Lemma_1*}
Let $c$ be a real number. Let $n,m\to+\infty$. Assume that conditions of Theorem \ref{Theorem_2}
hold. Assume that $\lambda^*_{n,m}\to c$. Then
\begin{equation}
\label{component2*}
\PP\{G_{[n,m]} \rm{\ has\  a \ component\  on} \  k\ \rm{vertices \ for \ some} \  2\le k\le n/2\}=o(1).
\end{equation}   
\end{lem}
\begin{proof}[Proof of Lemma \ref{Lemma_1*}]
The proof is similar to that of Lemma \ref{Lemma_1} and we use notation introduced there. Several minor modifications needed are explained in detail.
 
 Let $\beta=\max\{ 1-\frac{\alpha}{2\varkappa},\frac{1}{2}\}$. We  upper bound the probability $\PP\{{\cal A}_2\}$ by the sum 
$\PP\{{\cal A}_2\}
\le
S_0+S_1+S_2$, where $S_0,S_1,S_2$ are the same as in the proof of Lemma \ref{Lemma_1}, and show that $S_j=o(1)$ for $j=0,1,2$. The sequence
$\varphi_n\to+\infty$ as $n\to+\infty$ entering the expressions of $S_0$, $S_1$ above will be specified later.

Before the proof we collect some useful facts. 
We note that our assumption $\lambda^*_{n,m}\to c$ together with (\ref{2023-10-28}) implies 
$m=(\varkappa^{-1}+o(1))n\ln n$. In particular, we have
$m=O(n\ln n)$. 

Next, given 
$U\subset {\cal V}=[n]$ of size  $|U|\le n^{\beta}$ and $H\subset [m]$, we evaluate the probability of the event  
${\cal I}(U,H)$ that none of the graphs 
$G_{n,i}$, $i\in H$ 
has an edge connecting some 
$v\in U$ and $w\in {\cal V}\setminus U$.
Using inequality
\begin{equation}
\nonumber
{\hat q}^{[n,i]}_k
\le 
1
-
\frac{k}{n}\kappa_{n,i}
+
R_1
+
\E\left(R_2(k,X_{n,i})h(X_{n,i},Q_{n,i})\right),
\end{equation}
which follows from  (\ref{2023-05-01+1}), we obtain
for $U=[k]$ and any $H\subset [m]$ of size $h=|H|$ 
that
\begin{equation}\label{2023-10-30}
\PP\{{\cal I}([k],H)\}
=
\prod_{i\in H}{\hat q}_k^{[n,i]}
\le 
e^{-\frac{k}{n}S_1(H)
+
hR_1
+
S_2(H)}.
\end{equation}
Here we denote
\begin{displaymath}
S_1(H)=\sum_{i\in H}\kappa_{n,i},
\qquad 
S_2(H)=\sum_{i\in H}\E\left(R_2(k,X_{n,i})h(X_{n,i},Q_{n,i})\right).
\end{displaymath}
Let us evaluate the second and third term in the argument of the exponent in (\ref{2023-10-30}).
 We have 
\begin{displaymath}
S_2(H)\le S_2([m])=
m\E\left(R_2(k,X_{n,i_*})h(X_{n,i_*},Q_{n,i_*})\right)
\le 2k\frac{m}{n\ln n}.
\end{displaymath} 
In the last step we used the bound shown in formula {\color{red}\bf(36)} of 
\cite{Daumilas_Mindaugas2023}
\begin{displaymath}
\E \left(R_2(k,X_{n,i_*})h(X_{n,i_*},Q_{n,i_*})\right)
\le \frac{k}{n}\frac{2}{\ln n},
\end{displaymath}
which holds uniformly in $1\le k\le n^{\beta}$
for all sufficiently large  $n$ (say $n>n_\beta$, where $n_\beta$ depends on $\beta$ and the probability distribution of the array of random variables $\{X_{n,i}, 1\le i\le m(n), \, n=1,2,\dots\}$, see \cite{Daumilas_Mindaugas2023}). 
Furthermore, for $m=O(n\ln n)$ and $1\le k\le n^{\beta}$ we have 
 \begin{displaymath}
 hR_1\le mR_1
 =
 m\frac{k^2}{(n-k)^2}
 =
 \frac{mk^2}{n^2(1-o(1))}
 =
 O\left(k^2\frac{\ln n}{n}\right)
 =
 O\left(k\frac{\ln n}{n^{1-\beta}}\right).
\end{displaymath}
Invoking in 
(\ref{2023-10-30})
the bounds for $S_2(H)$ and $hR_1$  and using $m=(\varkappa^{-1}+o(1))n\ln n$ we obtain
\begin{equation}
\label{2023-10-30+2}
\PP\{{\cal I}([k],H)\}
\le 
e^{-\frac{k}{n}S_1(H)
+
2k\frac{m}{n\ln n}+o(k)
}
=
e^{-\frac{k}{n}S_1(H)
+
k
\left(\frac{2}{\varkappa}+o(1)\right)
}
.
\end{equation}

{\it Proof of} $S_0=o(1)$.
We first note that  for
$m=O(n\ln n)$
inequalities (\ref{2023-10-25})
imply
\begin{equation}
\label{max}
K_n:=\max_{i\in[m]}\kappa_{n,i}=o(n).
\end{equation}
Hence, there exists an integer  sequence 
$\tau_n\to+\infty$ such that $\tau_n^2K_n\le n$.
Consequently, for any $H\subset [m]$ of size 
 $|H|\ge m-\tau_n$  we have 
\begin{displaymath}
S_1(H)
\ge S_1([m])-\tau_nK_n
\ge S_1([m])-\frac{n}{\tau_n}
=m\kappa_n-\frac{n}{\tau_n}
\end{displaymath}
and, by (\ref{2023-10-30+2}),
\begin{equation}
\label{2023-10-30+3}
\PP\{{\cal I}([k],H)\}
\le 
e^{-k\frac{m}{n}\kappa_n
+
k
\left(\frac{2}{\varkappa}+o(1)\right)
}
=e^{-k\ln n+k\left(\frac{2}{\varkappa}+c+o(1)\right)}.
\end{equation}

We choose $\varphi_n$ such  that $\varphi_n\le \tau_n$ and upper bound 
$S_0=\sum_{2\le k\le \varphi_n}
\binom{n}{k}\PP\{{\cal B}_{[k]}\}$ 
 similarly as in the proof of Lemma \ref{Lemma_1}, see (\ref{suma}),
 (\ref{suma++++}) above. In view of 
 (\ref{2023-10-30+3}) inequality
  (\ref{suma++++}) now reads as follows
 \begin{equation}
 \label{2023-10-30+6}
 \PP\{{\cal B}_{[k]}\}
 \le 
 e^{-k\ln n+k\left(\frac{2}{\varkappa}+c+o(1)\right)} 
 \sum_{T\in{\mathbb T}_k}S_T,
 \qquad
 S_T=\sum_{({\tilde {\cal E}},{\tilde t})}
 \PP\{{\cal T}({\tilde {\cal E}},{\tilde t})\}.
 \end{equation}
We recall that ${\mathbb T}_k$ stands for the set of all spanning trees with the vertex set $[k]$,    
${\tilde {\cal E}}_T=({\tilde {\cal E}}_T^{(1)},\dots, {\tilde {\cal E}}_T^{(r)})$ denotes an ordered partition of the edge set of a tree $T$,
and the event ${\cal T}({\tilde {\cal E}},{\tilde t})$ is defined in the proof of Lemma
\ref{Lemma_1} above. We show below that
\begin{equation}
\label{2023-10-30+5}
S_T
\le 
\frac{1}{\phi(n)}
\sum_{r=1}^{k-1}
\left(\frac{2\zeta_*}{\varkappa}\right)^r
r^{k-1}.
\end{equation}
 Invoking this inequality in (\ref{2023-10-30+6})
 and using Cayley's formula $|{\mathbb T}_k|=k^{k-2}$
 we obtain
\begin{displaymath}
\PP\{{\cal B}_{[k]}\}
 \le 
 e^{-k\ln n+k\left(\frac{2}{\varkappa}+c+o(1)\right)} 
 \frac{R_k}{\phi(n)},
 \qquad
 R_k:=k^{k-2}\sum_{r=1}^{k-1}
\left(\frac{2\zeta_*}{\varkappa}\right)^r
r^{k-1}.
\end{displaymath}
Hence
\begin{displaymath}
S_0\le \frac{1}{\phi(n)}
\sum_{k=2}^{\varphi_n}
e^{k\left(\frac{2}{\varkappa}+c+o(1)\right)}
\frac{R_k}{k!}.
\end{displaymath}
Finally, since $\phi(n)\to+\infty$ as $n\to+\infty$, we can choose $\varphi_n\to+\infty$ growing slowly enough such that
the quantity on the right was $o(1)$. We remark that condition 
$\varphi_n\le \ln n$ used in the proof of Lemma \ref{Lemma_1} is 
now replaced now by the condition $\varphi_n\le \tau_n$.

It remains to prove (\ref{2023-10-30+5}).
Given an ordered partition 
${\tilde {\cal E}}_T=({\tilde {\cal E}}_T^{(1)},\dots, {\tilde {\cal E}}_T^{(r)})$ we denote by 
$|{\tilde {\cal E}}_T|$ the number of parts (in this case $|{\tilde {\cal E}}_T|=r$).
For ${\tilde t}=(t_{i_1},\dots, t_{i_r})$ we denote by $|{\tilde t}|$ the number of coordinates (in this case $|{\tilde t}|=r$).
We have 
\begin{align*}
S_T
=
\sum_{({\tilde {\cal E}}_T,{\tilde t})}
\PP\{{\cal T}({\tilde {\cal E}}_T,{\tilde t})\}
=
\sum_{r=1}^{k-1}
\
\frac{1}{r!}
\sum_{{\tilde {\cal E}}_T:\,|{\tilde {\cal E}}_T|=r}
S({\tilde {\cal E}}_T),
\qquad
S({\tilde {\cal E}}_T)
:=
\sum_{{\tilde t}:\, |{\tilde t}|=|{\tilde{\cal E}}_T|}
\PP\{{\cal T}({\tilde {\cal E}}_T,{\tilde t})\}.
\end{align*}
The last sum runs over the set of 
vectors ${\tilde t}=(t_1,\dots t_r)$ having distinct coordinates $t_1,\dots, t_r\in [m]$.
Proceeding as in the proof of Lemma 
\ref{Lemma_1}
we show  for 
${\tilde {\cal E}}_T=({\tilde {\cal E}}_T^{(1)},\dots, {\tilde {\cal E}}_T^{(r)})$
with $|{\tilde {\cal E}}_T^{(1)}|=e_1,
\dots, |{\tilde {\cal E}}_T^{(r)}|=e_r$ that
\begin{displaymath}
\PP\{{\cal T}({\tilde {\cal E}}_T,{\tilde t})\}
\le
\prod_{i=1}^r
\E
\left(
\frac{(X_{n,t_i})_{e_i+1}}{(n)_{e_i+1}}Q_{n,t_i}^{e_i}
\right).
\end{displaymath}
Consequently, the sum $S({\tilde{\cal E}}_T)$ is upper bounded by the sum
\begin{displaymath}
S_*({\tilde{\cal E}}_T)
:=
\sum_{t_1=1}^m\cdots\sum_{t_r=1}^m
\prod_{i=1}^r
\E
\left(
\frac{(X_{n,t_i})_{e_i+1}}{(n)_{e_i+1}}Q_{n,t_i}^{e_i}
\right)
=m^r
\prod_{i=1}^r
\E
\left(
\frac{(X_{n,i_*})_{e_i+1}}{(n)_{e_i+1}}Q_{n,i_*}^{e_i}
\right).
\end{displaymath}
Invoking the bound (which is shown similarly as (\ref{XQ}) above)
\begin{displaymath}
\E \left(X_{n,i_*}^{a+1}Q_{n,i_*}^a\right)
\le 
\frac{n^a}{\phi(n)\ln(1+n)}\zeta_*,
\quad
a=1,2,\dots,
\end{displaymath}
we obtain
\begin{displaymath}
S_*({\tilde {\cal E}}_T)
\le 
m^r
\left(\frac{\zeta_*}{n\phi(n)\ln(1+n)}\right)^r.
\end{displaymath}
Next,
using the fact that there are $\frac{(k-1)!}{e_1!\cdots e_r!}$ distinct ordered partitions
${\tilde {\cal E}}_T=({\tilde {\cal E}}_T^{(1)},\dots, {\tilde {\cal E}}_T^{(r)})$
with $|{\tilde {\cal E}}_T^{(1)}|=e_1,
\dots, |{\tilde {\cal E}}_T^{(r)}|=e_r$, we upper bound
\begin{align*}
S_T
&
\le 
\sum_{r=1}^{k-1}
\sum'_{e_1+\cdots+e_r=k-1}
\frac{(k-1)!}{e_1!\cdots e_r!} m^r 
\left(\frac{\zeta_*}{n\phi(n)\ln(1+n)}\right)^r
\\
&
\le 
\sum_{r=1}^{k-1}
m^r 
\left(\frac{\zeta_*}{n\phi(n)\ln(1+n)}\right)^r
r^{k-1}.
\end{align*}
Here the sum $\sum'_{e_1+\dots+e_r=k-1}$ runs over the set of vectors
$(e_1,\dots, e_r)$ having integer valued coordinates $e_i\ge 1$ 
satisfying $e_ 1+\cdots+e_r=k-1$. 
Invoking the inequality $\frac{m}{n\ln (1+n)}\le \frac{2}
{\varkappa}$ (which holds for large $n$) we obtain (\ref{2023-10-30+5}).

\bigskip

{\it Proof of} $S_1=o(1)$.
We observe that the event ${\cal D}_k$ defined in the proof of Lemma \ref{Lemma_1} and the event ${\cal I}([k],[m])$ are the same.
Hence,  by (\ref{2023-10-30+2}), (\ref{2023-10-30+3}),
\begin{displaymath}
\PP\{{\cal D}_k\}
=
\PP\{{\cal I}([k],[m])\}\le e^{-k\frac{m}{n}\kappa_n
+
k
\left(\frac{2}{\varkappa}+o(1)\right)
}
\le 
e^{-k\ln n+k\left(\frac{2}{\varkappa}+c+o(1)\right)}.
\end{displaymath}
We conclude that
\begin{displaymath}
S_1
=
\sum_{\varphi_n\le k\le n^{\beta}}
\binom{n}{k}
\PP\{{\cal D}_k\}
\le 
\sum_{k\ge \varphi_n}
\frac{e^{k\left(\frac{2}{\varkappa}+c+o(1)\right)}}{k!}
=o(1)
\qquad
{\text{as}}
\qquad 
n\to+\infty
\end{displaymath}
 because the series  $\sum
\frac{e^{k\left(\frac{2}{\varkappa}+c+1\right)}}{k!}$ converges and  $\varphi_n\to+\infty$.

\bigskip

{\it Proof of} $S_2=o(1)$. From (\ref{2023-05-03}) we have for $k\le n/2$
\begin{align*}
{\hat q}^{[n,i]}_k
&
=
\PP\{X_{n,i}\le 1\}
+
\E \left(q_k(X_{n,i},Q_{n,i}){\mathbb I}_{\{X_{n,i}\ge 2\}}\right)
\\
&
\le 
1-2\alpha_{n,i}\frac{k(n-k)}{n(n-1)}
\le 
1-\alpha_{n,i}\frac{k}{n}.
\end{align*}
Consequently,
\begin{displaymath}
\PP\{{\cal D}_k\}
=
\prod_{i\in [m]}
{\hat q}^{[n,i]}_k
\le 
e^{-\frac{k}{n}\sum_{i\in [m]}\alpha_{n,i}}
=
e^{-k\frac{m}{n}\alpha_n}.
\end{displaymath}
Next, proceeding as in (\ref{2023-10-16}) above, we estimate using $\frac{m}{n}=\frac{\ln n}{\varkappa}(1+o(1))$ and 
$\alpha_n=\alpha(1+o(1))$
\begin{displaymath}
\binom{n}{k}
\PP\{{\cal D}_k\}
\le 
e^{2k+(1-\beta)k\ln n-\alpha_nk\frac{m}{n}}
=
e^{-k\left(\frac{\alpha}{\varkappa}+\beta-1+o(1)\right)\ln n}
\le 
e^{-k\left(\frac{\alpha}{2\varkappa}+o(1)\right)\ln n}.
\end{displaymath}
In the last step we used $\beta\ge 1-\frac{\alpha}{2\varkappa}$. Finally, we have
 $S_2\le 
 \sum_{k\ge n^{\beta}}
 e^{-k\left(\frac{\alpha}{2\varkappa}+o(1)\right)\ln n}=o(1)$ as $n\to+\infty$.
 \end{proof}

\begin{lem}
\label{Lemma_2*} Let $c$ be a real number. let $n,m\to+\infty$. Assume that conditions of Theorem \ref{Theorem_2}
hold. Assume that $\lambda^*_{n,m}\to c$. Then the number of vertices of degree $0$ in $G_{[n,m]}$ converges in distribution to the Poisson distribution with the mean $e^{c}$.
\end{lem}
\begin{proof}[Proof of Lemma \ref{Lemma_2*}]
The proof is similar to that of Lemma \ref{Lemma_2}. The only (minor) difference is in the evaluation of the probability 
$p:=\PP\{d(v_1)=0,\dots,d(v_r)=0\}$ of (\ref{2023-10-24+2}).

Let us evaluate the probability $p$.
For $k=1,\dots, m$ let $d_k(v)$ denote the degree of $v$ in $G_{n,k}=({\cal V}_{n,k},{\cal E}_{n,k})$ and we put $d_k(v)=0$ for $v\notin {\cal V}_{n,k}$.
We denote $P_{r,k}=\PP\{d_k(v_1)=0,\dots, d_k(v_r)=0\}$.

By the independence of $G_{n,1},\dots, G_{n,m}$, we have 
$p=\prod_{k\in[m]}P_{r,k}$.  Furthermore, we have, by the inclusion-exclusion principle (cf. (\ref{X19})),
\begin{equation}
\label{2023-10-26}
S_{1,k}-S_{2,k}\le 1-P_{r,k}\le S_{1,k},
\end{equation}
where the sums $S_{1,k}$ and $S_{2,k}$ are defined as in the proof of Lemma \ref{Lemma_2}, but with $d_1(v_i), d_1(v_j)$ replaced by $d_k(v_i), d_k(v_j)$. Proceeding as in the proof of 
(\ref{X19+}), (\ref{2023-10-24+21}) we evaluate
$S_{1,k}= \kappa_{n,k}\frac{r}{n}$ and estimate 
\begin{displaymath}
S_{2,k}\le \frac{r^2}{2n^2}\Delta_{n,k},
\qquad
\Delta_{n,k}:=\E \left((X_{n,k})_2h(X_{n,k},Q_{n,k})\right).
\end{displaymath}
Now (\ref{2023-10-26}) yields the approximation
\begin{align}
P_{r,k}
=
1-\frac{r}{n}\kappa_{n,k}+\theta_{n,k}\frac{r^2}{2n^2}\Delta_{n,k}
=:1-a_k+b_k,
\end{align}
where $|\theta_{n,k}|\le 1$.
We note that 
\begin{equation}
\label{Delta+}
\max_{k\in [m]} \Delta_{n,k}
\le 
\sum_{k\in[m]}\Delta_{n,k}
=m\E\left((X_{n,i_*})_2h(X_{n,i_*},Q_{n,i_*})\right)=o(n^2),
\end{equation}
where the last bound follows from (\ref{2023-10-26+1}). Hence $b_k=o(1)$ uniformly in $k$ as $n\to+\infty$. Similarly, (\ref{max}) implies that $a_k=o(1)$ uniformly in $k$ as $n\to+\infty$.

Next,
using 
 $1+t=e^{\ln(1+t)}$ and inequality
 $t-t^2\le \ln(1+t)\le t$ valid for $|t|\le 0.5$ we write $P_{r,k}$ in the form
 \begin{align*}
 P_{r,k}=e^{\ln(1-a_k+b_k)}
 =
 e^{-a_k+b_k+R},
 \qquad
 {\text{where}}
 \qquad
 R\le (-a_k+b_k)^2\le 2a_k^2+2b_k^2.
 \end{align*}
 Furthermore, denoting $R_{r,k}:=b_k+R$ 
 we obtain 
 \begin{displaymath}
 P_{r,k}=e^{-a_k+R_{r,k}}=e^{-\frac{r}{n}\kappa_{n,k}+R_{r,k}},
 \end{displaymath}
 where $R_{r,k}$ satisfies
 \begin{displaymath}
 |R_{r,k}|
 \le 
 \frac{r^2}{2n^2}\Delta_{n,k}
 +
 2\left(\frac{r}{n}\kappa_{n,k}\right)^2
 +
2\left(\frac{r^2}{2n^2}\Delta_{n,k}\right)^2.
 \end{displaymath}
 Let us analyse the remainder term $R_{r,k}$.  By (\ref{Delta+}),  we have for large $n$ that
 $\max_{ k\in[m]}\frac{r^2}{n^2}\Delta_{n,k}<~1$.
Hence
\begin{displaymath}
 |R_{r,k}|
 \le 
 \frac{r^2}{n^2}\Delta_{n,k}
 +
 2\left(\frac{r}{n}\kappa_{n,k}\right)^2.
 \end{displaymath}
 Consequently
 \begin{displaymath}
 \left|\sum_{k\in [m]}R_{r,k}\right|
 \le
\frac{r^2}{n^2}\sum_{k\in [m]}\Delta_{n,k}
+
2\frac{r^2}{n^2}\sum_{k\in[m]}\kappa^2_{n,k}
=o(1).
\end{displaymath}
 In the last step we used (\ref{2023-10-25}) and the bound 
 $\sum_{k\in [m]}\Delta_{n,k}=o(n^2)$ of  (\ref{Delta+}).
 
 Finally, we evaluate the probability 
 \begin{displaymath}
 p
 =
 \prod_{k\in[m]}P_{r,k}
 =
 e^{-\frac{r}{n}\sum_{k\in[m]}\kappa_{n,k}+\sum_{k\in [m]}R_{r,k}}
 =
 e^{-r\frac{m}{n}\kappa_n+o(1)}.
 \end{displaymath}
The respective factorial moment
\begin{displaymath}
\E(Y_0)_r
=
(n)_rp
=
(1+o(1))e^{r\ln n-r\frac{m}{n}\kappa_n+o(1)}
=
(1+o(1))e^{r\lambda^*_{m,n}+o(1)}
\end{displaymath}
converges to $e^{rc}$ since $\lambda^*_{m,n}\to c$ as 
$n\to+\infty$.
\end{proof}

\end{document}